\documentclass[11pt,oneside,leqno]{amsart}
\usepackage{amsxtra}
\usepackage{amsopn}
\usepackage{amsmath,amsthm,amssymb}
\usepackage{amscd}
\usepackage{amsfonts}
\usepackage{latexsym}
\usepackage{verbatim}

\theoremstyle{plain}
\newtheorem{thm}{Theorem}[section]

\newtheorem{theorem}{Theorem}[section]
\newtheorem*{theorem*}{Theorem}

\newtheorem{lemma}[theorem]{Lemma}
\newtheorem{prop}[theorem]{Proposition}

\newtheorem*{mt*}{Main Theorem}
\theoremstyle{remark}
\newtheorem{remark}{Remark}
\newtheorem{example}{Example}
\newcommand\R{{\mathbb R}}

\newcommand\Z{{\mathbb Z}}
\newcommand\N{{\mathbb N}}
\newcommand\T{{\mathbb T}}

\newcommand{\Ad}{{\mathrm {Ad\,}}}
\newcommand{\Aut}{\mathrm{Aut\,}}
\newcommand{\Heis}{\mathrm{H}}

\newcommand{\ch}{{\mathfrak{h}}}

\newcommand{\ct}{{\mathfrak{t}}}

\newcommand{\mcn}{\multicolumn}
\newcommand \vr{\rule[-2mm]{0mm}{6mm}}

\begin{document}

\title[Solvable models for Kodaira surfaces]{Solvable models for Kodaira surfaces}
\author{Sergio Console}
\address{S. Console: Dipartimento di Matematica G. Peano \\ Universit\`a di Torino\\
Via Carlo Alberto 10\\
10123 Torino\\ Italy.} \email{sergio.console@unito.it}
\subjclass[2000]{53C30,22E25,22E40}

\author{Gabriela P. Ovando}
\address{G. Ovando: CONICET-FCEIA, U.N.R., Pellegrini 250, 2000 Rosario, Argentina.}
\email{gabriela@fceia.unr.edu.ar}
\author{Mauro Subils}

\address{M. Subils: C.I.E.M.-Fa.M.A.F., U.N.C., Ciudad Universitaria, 5000 C\'ordoba, Argentina.} \email{subils@famaf.unc.edu.ar}

\thanks{This work was supported by Secyt-UNR, ANCyT and CONICET,  by  MIUR and GNSAGA of INdAM
}

\begin{abstract} We consider  three families of lattices on the oscillator group $G$, which is an almost nilpotent not completely solvable Lie group, giving rise to coverings
$G \to M_{k, 0} \to M_{k, \pi} \to M_{k, \pi/2}$ for $k\in \Z$. We show that the corresponding families of four dimensional solvmanifolds are not pairwise diffeomorphic and we compute their cohomology and minimal models. In particular, each manifold $M_{k, 0}$  is diffeomorphic to a Kodaira--Thurston manifold, i.e. a compact quotient $S^1 \times \Heis_3 (\R) /\Gamma_k$ where $\Gamma_k$ is a lattice of the real three-dimensional Heisenberg group $\Heis_3 (\R)$. \\
We summarize some geometric aspects of those compact spaces. In particular, we note  that any $M_{k, 0}$ provides an example of a solvmanifold whose cohomology does not depend on the Lie algebra only and which admits many symplectic structures that are invariant by the group $\R \times\Heis_3 (\R)$ but not  under the oscillator group $G$. 
 \end{abstract}
\maketitle
\section{Introduction}


A solvmanifold $M$ 
is a compact homogeneous space of a solvable Lie group. In dimension four the so-called Kodaira surfaces are representable as $M=G/\Gamma$ where $G$ is  a connected and simply connected solvable Lie group and $\Gamma$ is  a lattice in $G$ (that is, a co-compact discrete subgroup of $G$).

\smallskip

In these notes we study models for some four dimensional solvmanifolds, which are known as Kodaira surfaces \cite{Hasegawa}. In fact both the primary and the secondary Kodaira  manifolds can be realized as quotients of a fixed solvable (non nilpotent)  Lie group $G$ of dimension four by different lattices.

We start be determining  three families of lattices in the  solvable Lie group called the {\em oscillator group}, which is   the semidirect product $G=\R  \ltimes_{\alpha} \Heis_3 (\R)$ of the (real) three dimensional Heisenberg group $\Heis_3(\R)$ by the map
\[
\alpha: \R \to \Aut (\ch_3)\, , \qquad
t \mapsto \left( \begin{array} {ccc} \cos (t)& \sin(t)& 0\\
 -\sin (t )& \cos(t)& 0\\
 0&0&1 \end{array}
 \right)\, .
\]

The oscillator group is an example of {\em almost nilpotent} solvable Lie group (see Section~\ref{Mostow}).

If we regard $\Heis_3(\R)$ as $\R^3$ endowed with the operation
$$
(x,y,z) \cdot (x',y',z')=(x+x',y+y', z+z'+ \frac 1 2 (xy'-x'y))
$$
then it admits the co-compact subgroups $\Gamma_k \subset \Heis_3(\R)$ given by
$$\Gamma_k=\Z \times \Z \times \frac 1 {2k} \Z\, .$$

The lattice $\Gamma_k$ (for any $k$) is invariant under the subgroups generated by $\alpha(0)=\alpha (2\pi)$, $\alpha (\pi)$ and $\alpha (\frac \pi 2)$. Consequently we have three families of lattices in $G=\R  \ltimes_{\alpha} \Heis_3 (\R)$:
\[
\begin{array}{l}
\Lambda_{k,0}=2\pi \Z \ltimes \Gamma_k \subset G\, ,\\
\Lambda_{k,\pi}=\pi \Z \ltimes \Gamma_k \subset G\, ,\\
\Lambda_{k,\pi/2}=\frac \pi 2 \Z \ltimes \Gamma_k \subset G\, .
  \end{array}
\]
so that $\Lambda_{k,0} \triangleright \Lambda_{k,\pi} \triangleright \Lambda_{k,\pi/2}$ (where $\triangleright$ means ``contains as a normal subgroup''), which induce the solvmanifolds
\[
\begin{array}{l}
M_{k,0}=G/\Lambda_{k,0}\, ,\\
M_{k,\pi}=G/\Lambda_{k,\pi}\, ,\\
M_{k,\pi/2}=G/\Lambda_{k,\pi/2}\, .
  \end{array}
\]

We prove that all subgroups of the families $\Lambda_{k,i}$ are not pairwise isomorphic, hence they determine non-diffeomorphic solvmanifolds.

Observe that the action of $\alpha(0)$ is trivial, so $\Lambda_{k,0}=2\pi \Z \times \Gamma_k$ and by Theorem~\ref{diffeo},
$M_{k,0}=G/\Lambda_{k,0}$ is diffeomorphic to $S^1 \times \Heis_3 (\R)/\Gamma_k$, a {\em Kodaira--Thurston manifold}.

Moreover, for any  fixed $k$, we have the finite coverings
\[
\begin{array}{l}
p_\pi: M_{k,0} \to M_{k,\pi}\, , \\
p_{\pi/2}: M_{k,0} \to M_{k,\pi/2}\, ,
 \end{array}
\]
which are 2- and 4-sheeted respectively and so

$$ G \longrightarrow M_{k,0} \longrightarrow M_{k,\pi} \longrightarrow M_{k,\pi/2}.$$

\medskip

In Section~\ref{cohom} we compute the cohomology of the minimal model of all solvmanifolds in the above families. 

\begin{theorem}\label{betti}  The Betti numbers $b_i$ of the solvmanifolds $M_{k,*}$ are given by
$$\begin{array}{l@{}c@{}l r@{}c@{}l r@{}c@{}lr@{}c@{}l r@{}c@{}l  c}
 \textrm{}    &&& \mcn{3}{c}{\vr b_0} & \mcn{3}{c}{\vr b_1} & \mcn{3}{c}{\vr b_2}  \\
\hline
  M_{k,0} &&& 1 &&& 3 &&& 4 \\
  M_{k,\pi} &&& 1 &&& 1 &&& 0 \\
  M_{k,\pi/2} &&& 1 &&& 1 &&& 0 \\
\hline
\end{array}$$
(clearly $b_3=b_1$ and $b_4=b_0$, by Poincar\'e duality).

A minimal model of $M_{k,0}$ is given by
$$
\mathcal M_{k,0} = (\Lambda (x_1, y_1, z_1, t_1), d)
$$
where the index denotes the degree (hence all generators have degree one) and the only non vanishing differential is given by $dz_1=-x_1 y_1$.

A minimal model of $M_{k,\pi}$ and $M_{k,\pi/2}$ is given by
$$
\mathcal M_{k,\pi} =\mathcal M_{k,\pi/2}=(\Lambda (t_1, w_3), d=0)
$$
where the index denotes the degree, cf. \cite{OT, OT2}.
\end{theorem}

Concerning  the geometry the space $M_{k, 0}$ provides an example of solvmanifold which admits symplectic structures but no invariant ones. Actually, in Section~\ref{sympl_str} we prove the following 

\begin{theorem} \label{solv} There are many symplectic structures on $M_{k, 0}$ which are invariant by the group $  \R \times \Heis_3 (\R)$ but not under the oscillator group $G$.
\end{theorem}

Moreover $M_{k, 0}$ covers $M_{k,\pi}$ and $M_{k,\pi/2}$ which do not admit any symplectic structure, since their second Betti number vanishes.

It is known that if a given nilmanifold $N/\Gamma$ admits a symplectic structure, then it admits an $N$-invariant  one.
Hence we provide low dimensional examples which show that this is not true for solvmanifolds.

Observe that any $M_{k, 0}$ gives a (low dimensional) example of solvmanifold whose de Rham cohomology does not agree with the invariant one, i.e. the cohomology of the Chevalley--Eilenberg complex on the solvable Lie algebra (unlike the case of nilmanifolds and solvmanifolds in the completely solvable case \cite{Hattori} and, more generally, for which the Mostow condition holds  \cite{Raghunathan, Guan, CF2}, cf. Section~\ref{cohom}).

Actually, it turns out that $M_{k, 0}$ has the same cohomology as a nilmanifold, namely the Kodaira--Thurston manifold $S^1 \times \Heis_3 (\R)/\Gamma_k$, cf. Section~\ref{cohom}. Moreover, passing from $M_{k, 0}$ to the covered manifolds $M_{k,\pi}$ and $M_{k,\pi/2}$, the cohomology changes. So the cohomology depends on the lattice and not on the solvable Lie algebra only.

\section{Three Families of Solvmanifolds}

Unlike the special case of nilmanifolds (i.e., compact quotients  of nilpotent Lie groups by a lattice), there is no simple criterion for the existence of a lattice in a connected and simply-connected solvable Lie group.
A necessary condition is that the connected and simply-connected solvable Lie group is unimodular {\cite[Lemma 6.2]{Mil}}.


Lattices determine the topology of  compact solvmanifolds since they are Eilenberg--MacLane space of type $K(\pi ,1)$ (i.e. all homotopy groups vanish, besides the first) with finitely
generated torsion-free fundamental group.
Actually lattices associated to solvmanifolds yield their diffeomorphism class as the following theorem states.


\begin{theorem} \cite[Theorem 3.6]{Raghunathan}  \label{diffeo}
Let $G_i / \Gamma_i$ be solvmanifolds for $i \in \{1,2\}$ and
$\varphi : \Gamma_1 \to \Gamma_2$ an isomorphism. Then there exists a diffeomorphism $\Phi : G_1 \to G_2$ such that
\begin{itemize}
\item[(i)] $\Phi|_{\Gamma_1} = \varphi ,$
\item[(ii)] $\Phi(p \gamma) = \Phi(p) \varphi(\gamma)$, for any $\gamma \in \Gamma_1$ and any $p \in G_1$.
\end{itemize}
\end{theorem}

As a consequence  two solvmanifolds with isomorphic fundamental groups are diffeomorphic.

\smallskip

Recall that if the action of the  group $\Gamma$ on the topological space $Y$ is properly discontinuous, then there is a differentiable structure on $Y/\Gamma$ such that  $Y \to Y/\Gamma$ is a normal covering, $\Gamma$ is the Deck transformation group of the covering and if $Y$ is simply connected then $\Gamma$ is isomorphic to $\pi_1(Y/\Gamma)$ \cite[Proposition 1.40]{Hatcher}.

Each discrete subgroup $\Lambda_{k,i}$ $i=0,\pi, \pi/2$ acts properly discontinuous on the simply connected space $G$. In the sequel we shall see that  they are pairwise non isomorphic so that the resulting compact spaces are pairwise non homeomorphic.

\begin{lemma} \label{center}
Let $k\in\N$. Then

\begin{enumerate}

\item  $Z_{k}=2\pi \Z \times 0 \times 0 \times \frac 1 {2k} \Z$ is the center of $\Lambda_{k,i}$ for $i=0, \pi, \pi/2$ and

\item $C= 0 \times 0 \times 0 \times \Z$ is the commutator of $\Lambda_{k,0}$.
\end{enumerate}

\end{lemma}

\begin{proof}
Fix $k\in\N$ and take $i=0, \pi, \pi/2$. By a simple computation we see that for each $k$ the set $Z_{k}$ is contained in the center of $G$, and then it is contained in the center of $\Lambda_{k,i}$.

Now, let $(\theta,a,b,c)$ be in the center of $\Lambda_{k,i}$, then
\[
\begin{array}{c}
(\theta,a,b,c)(0,1,0,0)=(0,1,0,0)(\theta,a,b,c) , \\
(\theta, a+\cos \theta, b-\sin \theta, c- \frac 1 {2}(a\sin \theta + b\cos \theta))=(\theta, 1+a, b, c+ \frac 1 {2}b)\, ,
 \end{array}
\]

It follows that $\theta=2l\pi$ and $b=0$.

Also from $(2l\pi,a,0,c)(0,0,1,0)=(0,0,1,0)(2l\pi,a,0,c)$ we get $a=0$. Then $(\theta,a,b,c)\in Z_{k}$

Now we prove that $C$ is the commutator of $\Lambda_{k,0}$. By computing we see
$$
(2l\pi,a,b,c)(2l'\pi,a',b',c')(2l\pi,a,b,c)^{-1}(2l'\pi,a',b',c')^{-1} = (0,0,0,ab'-a'b)\in C
$$
Since $C$ is a subgroup, we have that the commutator is contained in $C$. But taking $(2l\pi,a,b,c)=(0,x,0,0)$ and $(2l'\pi,a',b',c')=(0,0,1,0)$ for any $x\in\Z$, it follows that the element of $C$ given by $(0,0,0,x)$ belongs to the commutator and this completes the proof.
\end{proof}

\begin{prop} \label{nonisom}
The groups $\Lambda_{k,i}$, $k\in \mathbb N$, $i=0, \pi/2, \pi$, are pairwise not isomorphic.
\end{prop}

\begin{proof}
First we observe that if $\varphi:\Lambda_{p,j} \to \Lambda_{k,i} $ is an isomorphism  then $\varphi((2l+1)\pi,a,b,c)\neq(2l'\pi,a',b',c')$ for $l,l'\in\Z$. Otherwise, we get
$$
\varphi((4l+2)\pi,0,0,z) = \varphi((2l+1)\pi,a,b,c)^{2} = (2l'\pi,a',b',c')^{2} = (4l'\pi,2a',2b',z') \in Z_{k}\\
$$
which is the center of $\Lambda_{k,i}$  by Lemma~\ref{center}; it follows that  $a'=b'=0$. So $\varphi((2l+1)\pi,a,b,c)\in Z_{k}$  and $((2l+1)\pi,a,b,c)\in Z_{p}$, which is a contradiction.
Considering $\varphi^{-1}$ we get also that $\varphi(2l\pi,a,b,c)\neq((2l'+1)\pi,a',b',c')$ for $l,l'\in\Z$.
We conclude that $\Lambda_{k,0}$ is not isomorphic to $\Lambda_{p,\pi}$ nor to $\Lambda_{p,\pi/2}$.

If there is an isomorphism $\varphi_{1}:\Lambda_{k,\pi/2}\rightarrow \Lambda_{p,\pi}$ with $\varphi_{1}(\pi/2,0,0,0)=(l\pi,a,b,c)$, then:
$$
\varphi_{1}(\pi/2,0,0,0)^{2} =(l\pi,a,b,c)^{2}\Rightarrow \varphi_{1}(\pi,0,0,0)=(2l\pi,x,y,z),
$$
and we show that this cannot happen.

Suppose that $\varphi_{2}:\Lambda_{k,0}\rightarrow \Lambda_{p,0}$ is an isomorphism with $p < k$. By Lemma~\ref{center}, $\varphi_{2}(C)=C$.
$$
\varphi_{2}(0,0,0,1) = \varphi_{2}(0,0,0,\frac 1 {2k})^{2k} = (2\pi a,b,c,\frac d {2p})^{2k} = (4\pi ka,2kb,2kc,\frac e {2p}),\\
$$
for some $a,b,c,d,e\in\Z$. Then $(2ka,2kb,2kc,\frac e {2k})\in C$ and $a=b=c=0$.
$$
\varphi_{2}(0,0,0,\frac p {k}) = \varphi_{2}(0,0,0,\frac 1 {2k})^{2p} = (0,0,0,\frac d {2p})^{2p} = (0,0,0,d)\in C,\\
$$
So we have $(0,0,0,\frac p {k})\in C$. Absurd.

Let $\varphi_{3}:\Lambda_{k,\pi}\rightarrow \Lambda_{p,\pi}$ be an isomorphism. 
By the remark at the beginning of the proof, we conclude that the restriction of $\varphi_{3}$ to $\Lambda_{k,0}$ is an isomorphism from $\Lambda_{k,0}$ to $\Lambda_{p,0}$ which contradicts the last paragraph.



Now, let $\varphi_{4}:\Lambda_{k,\pi/2}\rightarrow \Lambda_{p,\pi/2}$ be an isomorphism. If $\varphi_{4}((2l+1)\pi/2,a,b,c)=(l'\pi,a',b',c')$, then
$\varphi_{4}((2l+1)\pi/2,a,b,c)^{2} =(l'\pi,a',b',c')^{2}$ and follows that $\varphi_{4}((2l+1)\pi,x,y,z)=(2l'\pi,0,0,z')$ which also contradicts the remark. Then $\varphi_{4}(\Lambda_{k,\pi})=\Lambda_{p,\pi}$.

\end{proof}

\begin{thm} \label{nonisom2}
The subgroups $\Lambda_{k,i}$ are the only lattices of $G$ of the form $L_{1}\times L_{2}\times L_{3}\times L_{4}$ where $L_{i}\subset\R$ is a subset for  every $i=1,2,3,4$.
\end{thm}

\begin{proof}
Let $L= L_{1}\times L_{2}\times L_{3}\times L_{4}$ be a lattice of $G$, then it is easy to see that $L_{i}$ is a discrete subgroup of $\R$ for $i=1, 2, 3, 4$. Then there are $p,q,r,s\in \R_{\geq 0}$ such that
$$
L = p\Z \times q\Z \times r\Z \times s\Z
$$

Let $m \in\Z$, since $(0,q,0,0)\in L$, $(0,0,rm,0)\in L$ then $(0, q,rm,\frac {qr} {2} m)\in L$ and $\frac {qr}{2} m \in L_{4}$. It follows that $s=\frac {qr}{2k}$ for some $k\in \N$.

On the other hand, since $(p,0,0,0)\in L$ and  $(0,q,0,0)\in L$ then $(p, q\cos p, -q\sin p,0)\in L$ and $\cos p \in \Z$. It follows that $p= \frac {\pi}{2}l$ for some non negative integer $l$.

We conclude that:
$$
L = \frac {\pi}{2}l\Z \times q\Z \times r\Z \times \frac {qr}{2k}\Z
$$
for $q,r\in \R_{\geq 0}$, $k \in \N$ and $l$ a non negative integer.

If $l=0$, then $G/L \cong \R \times \Heis_3(\R)/L'$  which is not compact for $L' \subset \Heis_3(\R)$. If $r=0$ then $L \cap \Heis_3(\R) = q\Z \times 0 \times 0$  which is not a lattice in $\Heis_3(\R)$.  Analogously, if $q=0$.

Therefore $l,r,q$ are non zero real numbers, moreover $l\in \N$.

Now we consider four cases $l\equiv 0,1,2,3 \quad (\mbox{mod }4)$.

\begin{itemize}

\item $l \equiv 0 \quad (\mbox{mod  }4)$

A set of the form $L =  2\pi l\Z \times q\Z \times r\Z \times \frac {qr}{2k}\Z$ is a lattice of $G$ and it is isomorphic to $\Lambda_{k,0}$ via the isomorphism:

\[
\begin{array}{ll}
\gamma_{1}:&\Lambda_{k,0} \to  2\pi l\Z \times q\Z \times r\Z \times \frac {qr}{2k}\Z \\
&\gamma_{1}(t,x,y,z)= (lt, qx, ry, qrz)\,
  \end{array}
\]

\item $l \equiv 2 \quad (\mbox{mod }4)$

A set of the form $L =  (2l+1)\pi \Z \times q\Z \times r\Z \times \frac {qr}{2k}\Z$ is a lattice of $G$ which  is isomorphic to $\Lambda_{k,\pi}$ via the isomorphism:

\[
\begin{array}{ll}
\gamma_{2}:&\Lambda_{k,\pi} \to  (2l+1)\pi \Z \times q\Z \times r\Z \times \frac {qr}{2k}\Z \\
&\gamma_{2}(t,x,y,z)= ((2l+1)t, qx, ry, qrz)\,
  \end{array}
\]

\item $l \equiv 1 \quad (\mbox{mod } 4)$

Let $L=(4l+1)\frac{\pi} {2} \Z \times q\Z \times r\Z \times \frac {qr}{2k}\Z$ be a subgroup of $G$, then

\medskip

$\bullet$ $((4l+1)\frac{\pi} {2},0,0,0)(0,-q,0,0)=((4l+1)\frac{\pi} {2},0,q,0)\in L$ and

\medskip

$\bullet$ $((4l+1)\frac{\pi} {2},0,0,0)(0,0,r,0)=((4l+1)\frac{\pi} {2},r,0,0)\in L$.

 \medskip

Thus we deduce  that $r|q$ and $q|r$, so $q=r$.

 A set of the form $L =  (4l+1)\frac{\pi} {2} \Z \times q\Z \times q\Z \times \frac {q^{2}}{2k}\Z$ is a lattice of $G$ and it is isomorphic to $\Lambda_{k,\pi/2}$ via the isomorphism:

\[
\begin{array}{ll}
\gamma_{3}:&\Lambda_{k,\pi/2} \to  (4l+1)\frac{\pi} {2} \Z \times q\Z \times q\Z \times \frac {q^{2}}{2k}\Z \\
&\gamma_{3}(t,x,y,z)= ((4l+1)t, qx, qy, q^{2}z)\,
  \end{array}
\]

\item $l \equiv 3 \quad (\mbox{mod } 4)$

Let $L=(4l+3)\frac{\pi} {2} \Z \times q\Z \times r\Z \times \frac {qr}{2k}\Z$ be a subgroup of $G$, as before

\medskip

$\bullet$ $((4l+3)\frac{\pi} {2},0,0,0)(0,q,0,0)=((4l+3)\frac{\pi} {2},0,q,0)\in L$ and

\medskip

$\bullet$ $((4l+3)\frac{\pi} {2},0,0,0)(0,0,-r,0)=((4l+3)\frac{\pi} {2},r,0,0)\in L$

\medskip

which implies  $q=r$.

 The set $L = (4l+3)\frac{\pi} {2} \Z \times q\Z \times q\Z \times \frac {q^{2}}{2k}\Z$ is a lattice of $G$ which is isomorphic to $\Lambda_{k,\pi/2}$ via the isomorphism:
\[
\begin{array}{ll}
\gamma_{4}:&\Lambda_{k,\pi/2} \to  (4l+1)\frac{\pi} {2} \Z \times q\Z \times q\Z \times \frac {q^{2}}{2k}\Z \\
&\gamma_{4}(t,x,y,z)= ((4l+3)t, -qx, -qy, -q^{2}z)\,
  \end{array}
\]
\end{itemize}

\end{proof}

\begin{remark} Notice that there exist lattices in $G$ which are not of the  form $L_{1}\times L_{2}\times L_{3}\times L_{4}$. For instance, let $L$ be the next one
\begin{equation*}
    L = \left \{ \left(2l\pi ,2x,2y,\frac {1}{2}z\right ) : l,x,y,z\in\Z\right \}\cup\left \{\left ((2l+1)\pi ,2x+1,2y+1,\frac {1}{2}z\right ) : l,x,y,z\in\Z\right \}.
\end{equation*}
It contains the lattice $2\pi \Z \times 2\Z \times 2\Z \times 2\Z$. But this lattice is isomorphic to $\Lambda_{2,\pi}$  via the isomorphism $(t,x,y,z)\to(t,\frac {x-y}{2}, \frac {x+y}{2}, \frac {z}{2})$. We conjecture that every lattice of $G$ is isomorphic to one of the family $\Lambda_{k,i}$ as above .
\end{remark}

\section{The Mostow bundle and almost nilpotent Lie groups}\label{Mostow}

Let $M=G / \Gamma$ be a solvmanifold that is not a nilmanifold.
Let $N$ be the nilradical of $G$, i.e., the largest connected nilpotent normal subgroup of $G$.

Then $\Gamma_N := \Gamma \cap N$ is a lattice in $N$, $\Gamma N = N
\Gamma$ is  closed in $G$ and $G / (N \Gamma)=:\T^k$ is a torus.
Thus we have the so-called {\em Mostow fibration}:
$$ N / \Gamma_N = (N \Gamma) / \Gamma \hookrightarrow G / \Gamma
\longrightarrow G / (N \Gamma) = \T^k$$

Most of the rich structure of solvmanifolds is encoded in this bundle.

The fundamental group $\Gamma$ of  $M$ can be represented as an extension of a torsion-free nilpotent group $\Lambda$ of rank $n-k$ by a free abelian group of rank $k$ where $1 \leq k \leq 4$:
\begin{equation}\label{fg}
0 \longrightarrow \Lambda \longrightarrow \Gamma \longrightarrow \Z^k \longrightarrow 0
\end{equation}

The classification of solvmanifolds of dimension four reduces to the classification of the groups $\Gamma$ as the group extensions above (see \cite{Hasegawa}).

A connected and simply-connected solvable Lie group $G$ with
nilradical $N$ is called \emph{almost nilpotent} if its nilradical has codimension one. In this case $G$  can be written as a semidirect product $G = \mathbb{R}
\ltimes_{\mu} N$. In addition, if $N$ is abelian, i.e.\ $N =
\mathbb{R}^n$, then $G$ is called \emph{almost
abelian}.

Let $G = \mathbb{R} \ltimes_{\mu} N$ be an almost nilpotent Lie
group. Since $N$ has codimension one in $G$, we can consider $\mu$
as a one-parameter group $\mathbb{R} \to \mathrm{Aut}(N)$. Observe that $d\mu=:\phi$ is one-parameter subgroup of the automorphism group of the Lie algebra  $\mathfrak{n}$ of $N$.

\begin{example} Let consider the $3$-dimensional solvable Lie group $\R \ltimes \R^2$ with structure equations
 $$
 \left \{  \begin{array} {l}
d e^1 =0,\\
d e^2 = 2 \pi e^{13},\\
d e^3 = - 2 \pi e^{1 2}.
\end{array} \right .
 $$
 is a non-completely solvable Lie group which admits a compact quotient and the uniform discrete subgroup is of the form
  $\Gamma =  \Z \ltimes \Z^2$ (see \cite[Theorem 1.9]{OT} and \cite{Mi}). Indeed, the Lie group $\R \ltimes \R^2$   is the group of matrices
 $$
 \left( \begin{array} {cccc} \cos (2 \pi t )& \sin(2 \pi t)& 0 & x\\
 -\sin (2 \pi t )& \cos(2 \pi t)& 0 & y\\
 0&0&1&t\\
 0&0&0&1 \end{array}
 \right)
 $$ and the lattice $\Gamma$  generated by $1$ in $\R$  and the standard lattice $\Z^2$.
 The semidirect product is relative to the one-parameter subgroup
 $$
 t \mapsto \left( \begin{array} {ccc} \cos (2 \pi t )& \sin(2 \pi t)& 0\\
 -\sin (2 \pi t )& \cos(2 \pi t)& 0\\
 0&0&1 \end{array}
 \right)\, .
 $$
 The solvable Lie group $\R \ltimes \R^2$ is almost abelian.
\end{example}

\begin{example}  The oscillator group $G=\R  \ltimes_{\alpha} \Heis_3 (\R)$ is almost nilpotent. 
\end{example}

\section{Computation of Cohomology and Minimal Model}\label{cohom}

Let $M/\Gamma$ be a solvmanifold.
If the algebraic closures ${\mathcal A} (\Ad_G (G))$ and  ${\mathcal A} (\Ad_G (\Gamma))$ are equal, one says that $G$ and $\Gamma$ satisfy the {\em Mostow condition}. In this case the de Rham  cohomology $H^*(M)$ of  the compact  solvmanifold $M = G / \Gamma$  can be computed  by the  Chevalley--Eilenberg cohomology $H^* ({\mathfrak g})$ of the Lie algebra ${\mathfrak g}$ of $G$ (see \cite{Mostow2} and \cite[Corollary 7.29]{Raghunathan}); indeed, one has the isomorphism $H^*(M) \cong H^* ({\mathfrak g})$.
A special case is provided by nilmanifolds (Nomizu's Theorem, \cite{Nomizu}) and more generally if $G$ is completely solvable (\cite{Hattori}), i.e. all the linear operators ${\mbox {ad}}_X: {\mathfrak g} \to {\mathfrak g}$, $X \in {\mathfrak g}$  have only real eigenvalues.

\smallskip

Let us consider $M_{k,0}=G/\Lambda_{k,0}$.
Now, applying  similar methods as in \cite{Gor}, one can show that the algebraic closure  ${\mathcal A} (\Ad_G (G))$ of $\Ad_G (G)$ is $S^1 \times \Heis_3 (\R)$ and the one of $\Ad_G (\Lambda_{k,0})$ is $\Heis_3 (\R)$.
Thus the Mostow condition does not hold.

\smallskip

Actually to compute the cohomology of $M_{k,0}$ one can remark that it is diffeomorphic to $S^1 \times \Heis_3 (\R)/\Gamma_k$, the {\em Kodaira--Thurston manifold}. Hence we can easily write down its cohomology classes, in terms of the ones of $S^1$ and $\Heis_3 (\R)/\Gamma_k$.

By Nomizu's Theorem, the cohomology of $\Heis_3 (\R)/\Gamma_k$ is given by the Chevalley--Eilenberg cohomology $H^* ({\mathfrak h}_3)$ of the Lie algebra ${\mathfrak h}_3$ of $\Heis_3 (\R)$. By the structure equations
$$d\alpha=0,\qquad  d\beta=0,\qquad d\gamma=-\alpha \beta$$
 it follows that

 \begin{itemize}
\item $H^1(\Heis_3 (\R)/\Gamma_k)$ is generated by $\alpha, \beta$,

\item $H^2(\Heis_3 (\R)/\Gamma_k)$ is generated by $\alpha \gamma, \beta \gamma$ and

\item $H^3(\Heis_3 (\R)/\Gamma_k)$ is generated by $\alpha \beta \gamma$.

\end{itemize}

Let $\tau$ be the generator of $H^1(S^1)\cong \R^*$

\medskip

Thus the de Rham cohomology classes of $M_{k,0}=G/\Lambda_{k,0}$ are given by:
\begin{itemize}
\item $H^1 (M_{k,0})\cong \R^3$ is generated by $\tau, \alpha, \beta$.
\item $H^2 (M_{k,0})\cong \R^4$ is generated by $\tau \alpha, \tau \beta, \alpha \gamma, \beta \gamma$.
\item $H^3 (M_{k,0})\cong \R^4$ is generated by $\tau \alpha \gamma, \tau \beta \gamma, \alpha \beta \gamma$ .
\item $H^4 (M_{k,0})\cong \R^4$ is generated by $\tau \alpha \beta \gamma$.
\end{itemize}
A minimal model of $M_{k,0}$ is given by
$$
\mathcal M_{k,0} = (\Lambda (x_1, y_1, z_1, t_1), d)
$$
where the index denotes the degree (hence all generators have degree one) and the only non vanishing differential is given by $dz_1=-x_1 y_1$. It suffices to send $t_1$ to $\tau$, $x_1$ to $\alpha$ (and so on) to have a quasi isomorphism  $\mathcal M_{k,0} \to \Lambda M_{k,0}$, where $\Lambda M_{k,0}$ denotes the de Rham algebra of $M_{k,0}$.
This result can be also obtained  by applying the method in \cite{OT, OT2} for the Koszul-Sullivan model of the Mostow fibration.

\medskip

In order to compute the cohomologies of $M_{k,\pi}$ and $M_{k,\pi/2}$ recall that there are the 2-sheeted and 4-sheeted coverings $p_\pi: M_{k,0} \to M_{k,\pi}$ and $p_{\pi/2}: M_{k,0} \to M_{k,\pi/2}$.

\medskip

In general, if $q: X \to \tilde X$ is a finite sheeted covering defined by the action of a group $\Phi$ on $X$, then the cohomologies of $\tilde X$ are given by the invariants by the action of the finite group $\Phi$), i.e.
$$
H^* (\tilde X) \cong H^* (X)^{\Phi}\, ,
$$
(see e.g. \cite[Proposition 3G,1]{Hatcher}).

\medskip

Now, the (nontrivial part of the) action of $\Lambda_{k,\pi}/\Lambda_{k,0}$ is given by $\alpha \mapsto -\alpha$ and $\beta \mapsto -\beta$.

The (nontrivial part of the) action of $\Lambda_{k,\pi/2}/\Lambda_{k,0}$ is given by $\alpha \mapsto -\beta$ and $\beta \mapsto \alpha$. Computing the invariants, one easily sees that the de Rham cohomology of $M_{k,\pi/2}$ is the same as the one of $M_{k,\pi}$.

Thus the cohomology of $M_{k,\pi}$ and $M_{k,\pi/2}$ is
\begin{itemize} \label{cpi}
\item $H^1 (M_{k,\pi})\cong \R$ is generated by $\tau$.
\item $H^2 (M_{k,\pi})$ is trivial (there is no invariant 2-form).
\item $H^3 (M_{k,\pi})\cong \R$ is generated by $\alpha \beta \gamma$.
\item $H^4 (M_{k,\pi})\cong \R$ is generated by $\tau \alpha \beta \gamma$.
\end{itemize}

\smallskip
A minimal model of $M_{k,\pi}$ and $M_{k,\pi/2}$ is given by
$$
\mathcal M_{k,\pi} =\mathcal M_{k,\pi/2}= (\Lambda (t_1, w_3), d=0)
$$
where the index denotes the degree, cf. \cite[Example 3.2]{OT2}.
A quasi isomorphism $\mathcal M_{k,\pi} \to \Lambda M_{k,\pi}$, is given by $t_1 \mapsto \tau$ and $w_3 \mapsto \alpha \beta \gamma$.

The cohomologies of the Chevalley--Eilenberg complexes  of the Lie algebras ${\mathfrak g}$ of the oscillator group $G$ and of the nilpotent Lie algebra $\ct \times \ch_3$ of $S^1 \times \Heis_3 (\R)$ are given by:
$$\begin{array}{l@{}c@{}l r@{}c@{}l r@{}c@{}lr@{}c@{}l r@{}c@{}l r@{}c@{}lr@{}c@{}l c}
 \textrm{}    &&& \mcn{3}{c}{\vr b_0} & \mcn{3}{c}{\vr b_1} & \mcn{3}{c}{\vr b_2}& \mcn{3}{c}{\vr b_3}& \mcn{3}{c}{\vr b_4}  \\
\hline
{\mathfrak g} &&& 1 &&& 1 &&& 0 &&& 1 &&& 1 \\
\ct \times \ch_3 &&& 1 &&& 3 &&& 4 &&& 3 &&& 1\\
\hline
\end{array} $$
Hence, $M_{k, 0}$ has the same cohomology as a nilmanifold, namely the Kodaira--Thurston manifold $S^1 \times \Heis_3 (\R)/\Gamma_k$.
Thus any $M_{k, 0}$ gives a (low dimensional) example of solvmanifold whose cohomology does not agree with the invariant one, i.e. the cohomology of the Chevalley--Eilenberg complex on the solvable Lie algebra ${\mathfrak g}$.
On the other hand, the de Rham cohomologies of $M_{k,\pi}$ and $M_{k,\pi/2}$ are isomorphic  to the cohomology of the corresponding solvable Lie algebra ${\mathfrak g}$, although the Mostow condition does not hold.

\begin{remark}
In general, if the Mostow condition does not hold, as far as we know two techniques can be applied: the modification of the solvable Lie group \cite{Guan, CF2} and the cited Koszul-Sullivan  models of fibrations in the almost nilpotent case \cite{OT, OT2}.
As for the first method, one knows by Borel density theorem (see e.g. \cite[Theorem 5.5]{Raghunathan}) that there exists a compact torus $\T_{cpt}$ such that  $\T_{cpt}  {\mathcal A} (\Ad_G (\Gamma)) = {\mathcal A} (\Ad_G (G))$. Then one shows (see \cite{CF2}) that there exists a  subgroup $\tilde \Gamma$ of finite index in $\Gamma$ and a  simply connected   normal subgroup $\tilde G$ (the ``modified solvable Lie group'') of $\T_{cpt} \ltimes G$  such that  ${\mathcal A} (\Ad_{\tilde G}  (\tilde \Gamma)) = {\mathcal A} (\Ad_{\tilde G} ( \tilde G))$. Therefore,
 $\tilde G /  \tilde \Gamma$ is diffeomorphic to $G /  \tilde \Gamma$ and $H^* (G /  \tilde \Gamma) \cong H^*(\tilde {\mathfrak g})$, where $\tilde {\mathfrak g}$ is the Lie algebra of $\tilde G$.
In the case of the families $M_{k,0}$, $M_{k,\pi}$ and $M_{k,\pi/2}$ one sees that $\tilde \Gamma_k=\Lambda_{k,0}$ (for any $k$) and that $\tilde G_k$ is $S^1 \times \Heis_3 (\R)$.
\end{remark}

\section{About Complex and Symplectic structures}\label{sympl_str}
Here we shall study in more details the solvmanifolds as models of compact spaces provided with complex or symplectic structures.

For the classification of four-dimensional solvmanifolds (up to finite coverings) it is sufficient to classify the groups $\Gamma$  as group extensions as in \eqref{fg} and find a subgroup $\Gamma'$ which entends to a simply connected solvable Lie group $G$ such that $G/\Gamma'$ is a solvmanifold.

Due to results of Ue \cite{Ue} a complex surface $S$ is diffeomorfic to a $\T^2$ bundle over $\T^2$ if and only if $S$ is a complex torus, Kodaira surface or hyperelliptic surface. Moreover,  Hasegawa \cite{Hasegawa} proved

\begin{thm}
A complex surface is diffeomorphic to a four-dimensional solvmanifold if and only if it is one of the following surfaces: complex torus, hyperelliptic surface, Inou surface of type $S^0$, primary Kodaira surface, secondary Kodaira surface, Inoue surface of type $S^{\pm}$. And every complex structure on each of these complex surfaces is invariant.
\end{thm}

The invariance of the complex structure concerns the algebraic structure of the group covering the manifold. Thus for the so-called Kodaira--Thurston manifold, (above Kodaira surface of type I) we have two solvable groups which covers this space. On the one hand the oscillator group $G$ and on the other hand the nilpotent Lie group $\R \times \Heis_3(\R)$.

It is known that in the corresponding nilpotent Lie group $\R \times  \Heis_3 (\R)$ there is only one integrable almost complex structure up to equivalence but on $G$ there are two non-equivalent ones.

Furthermore, let $\R \times \mathfrak h_3 $ denote the Lie algebra of $\R \times \Heis_3 (\R)  $ and let $\mathfrak g$ denote the oscillator Lie algebra, that is the Lie algebra of the oscillator Lie group $G$.

Now, every complex structure $J$ on $ \R \times\mathfrak h_3 $ is abelian, that is $J$ satisfies $[Ju,Jv]=[u,v]$ for all $u,v$ in the Lie algebra. Every complex structure here is equivalent to (see \cite{Snow})

\begin{equation}\label{comp}
J\tilde{X}=\tilde{Y} \qquad \qquad J\tilde{Z}=\tilde{T} \qquad \qquad J^2=-1
\end{equation}
where $\tilde{X},\tilde{Y},\tilde{Z},\tilde{T}$ is a basis of left-invariant vector fields, specifically for the coordinates $(t,x,y,z) \in\R^4$ one has

$$
\begin{array}{rclrcl}
\tilde{T} & = & \frac{\partial}{\partial t} \qquad \qquad & \tilde{X} & = &  \frac{\partial}{\partial x} -\frac12 y \frac{\partial}{\partial z} \\ \\
\tilde{Y} & = &  \frac{\partial}{\partial y} +\frac12 x \frac{\partial}{\partial z} \qquad \qquad & \tilde{Z} & = & \frac{\partial}{\partial z} 
\end{array}
$$   
On the oscillator group the almost complex structure defined as in (\ref{comp}) is also integrable for the left-invariant vector fields $\tilde{X},\tilde{Y},\tilde{Z},\tilde{T}$   nevertheless it cannot be abelian. In fact, no complex structure on $\mathfrak g$ is abelian.  In this case the left-invariant vector fields are given by 

$$
\begin{array}{rclrcl}
\tilde{T} & = & \frac{\partial}{\partial t} \qquad &
 \tilde{X} & = &  \cos(t) \frac{\partial}{\partial x} -\sin(t)\frac{\partial}{\partial y} -\frac12 (x\cos(t) + y \sin(t)) \frac{\partial}{\partial z} \\ \\
\tilde{Z} & = & \frac{\partial}{\partial z} & \tilde{Y} & = &  \sin(t) \frac{\partial}{\partial x}+ \cos(t) \frac{\partial}{\partial y} +\frac12 (x \cos(t) - y\sin(t)) \frac{\partial}{\partial z}    
\end{array}
$$

Thus one has the next covering  as complex spaces

\begin{equation}\label{suc11} G \longrightarrow M_{k,0} \longrightarrow M_{k,\pi} \longrightarrow M_{k,\pi/2}.
\end{equation}

\begin{remark} (\ref{suc11}) above gives an explicit realization of the covering in Case (5) page 756 \cite{Hasegawa}.
\end{remark}

Concerning symplectic geometry notice that the oscillator group $G$ does not admit any invariant symplectic structure \cite{Ovando, Medina}.  But the Kodaira--Thurston manifold $M_{k, 0}$ was the first example constructed in order to provide an example of a compact manifold admitting a symplectic structure but no K\"ahler structures \cite{Th}. 

These remarks yield the proof of Theorem~\ref{solv}.

\smallskip

It is known that if a given nilmanifold $N/\Gamma$ admits a symplectic structure, then it admits an $N$-invariant  one. This follows by Nomizu's Theorem, since any  de Rham cohomology class has an invariant representative, and it is more in general true for solvmanifolds for which the Mostow condition holds.
The example above shows that this is not true for any solvable Lie group.

Moreover $M_{k,0}$ covers $M_{k,\pi}$ and $M_{k,\pi/2}$ which do not admit any symplectic structure, since their second Betti number vanishes.

\begin{remark}  \cite[Remarks 2 and 3]{Kasuya}: Another example of a non-symplectic manifold finitely covered by a symplectic manifold is constructed
\end{remark}

\begin{remark} {\em About metrics on the compact spaces $M_{k,i}$}. The oscillator group admits a pseudo-Riemannian metric so that the quotients $M_{k,i}$ constitute examples of compact naturally reductive Lorentzian spaces \cite{BOV}. In this case the geodesics are induced by the one-parameter groups of $G$.

On the other side  $\R\times \Heis_3(\R)$ is an example of a naturally reductive Riemannian space \cite{Kaplan}. 
\end{remark}

\medskip

{\bf Acknoledgements} The authors deeply thank A. Fino for useful discussions and comments on the paper. The second author is  grateful to the Department of Mathematics of the University of Torino, for the kindly hospitality during her stay at Torino, where part of the present work was done.

\end{document}